\documentclass[11pt]{amsart}

\usepackage{amssymb,amsfonts}
\usepackage{graphicx}	
\usepackage{hyperref}
\usepackage{verbatim}
\usepackage{amsmath}
 \usepackage{mathtools}
 \usepackage{xcolor}
 \usepackage{float}
 \usepackage{tikz}
\usetikzlibrary{matrix,arrows,decorations.pathmorphing}
 \usepackage{amsmath,amscd}
\usepackage[all, knot]{xy}
\usepackage{cite}
\usepackage{filecontents}

\usetikzlibrary{matrix}
\xyoption{arc}

\theoremstyle{definition}
\newtheorem{theorem}{Theorem}[section]

\newtheorem{lemma}[theorem]{Lemma}
\newtheorem{definition}[theorem]{Definition}

\newtheorem{corollary}[theorem]{Corollary}

\newtheorem{example}[theorem]{Example}

\newtheorem{notation}[theorem]{Notation}
\newtheorem{observation}[theorem]{Observation}

\newtheorem{construction}[theorem]{Construction}
\begin{document}

\title[Generalized edgewise subdivisions]{Generalized edgewise
subdivisions}
\author{Katerina Velcheva}
\date{}
\newcommand{\C}{\mathbb C}
\newcommand{\R}{\mathbb R}
\newcommand{\Z}{\mathbb Z}
\newcommand{\Q}{\mathbb Q}
\newcommand{\N}{\mathbb N}

\begin{abstract}
In this paper we classify endofunctors on the
simplex category, and we identify those that induce weak
equivalence preserving functors on the category of simplicial
sets.
\end{abstract}

\maketitle

\section{Introduction}

A \textit{generalized edgewise subdivision functor} is a weak
equivalence preserving endofunctor on the category of simplicial
sets that is induced by an endofunctor on the simplex category. In many
cases, it takes a general simplicial set  and
returns a \textit{nicer} one. Perhaps the most famous
generalized edgewise subdivision functor is Segal's subdivision\cite{segalsub}.

In this paper we classify all generalized edgewise subdivision
functors. 
\textit{Segal's edgewise subdivision functor} is the endofunctor
on the category of simplicial sets $SSet$ induced by the
concatenation of the opposite and the identity endofunctors on
the simplex category $\Delta$. Segal's edgewice subdivision is used, for example, to relate Quillen's Q-construction to  Waldhausen's S-construction. \cite{kt}

In Section 2 we recall the definitions of \textit{simplex
category} and the \textit{interval category} and we prove that
the two categories are dual.
 
In Section 3 we classify the endofunctors on the simplex
category by studying the functors that they induce on the
category of simplicial sets. The main idea of the classification
is that an endofunctor $T^*$ on $SSet$ that is induced by an
endofunctor $T$ on $\Delta$ is completely determined by its
value on the simplicial set $\Delta^1$. The duality between the
simplex category and the interval category allows us to regard
$T^*(\Delta^1)$ as a simplicial interval. A specific
interpretation of the structure of this simplicial interval
gives the desired classification.
 
The main result of the paper is proved in Section 4, where we
 classify the
endofunctors on the simplex category that induce  weak
equivalence preserving functors on the category of simplicial
sets. This answers a question posed to us by Clark Barwick. 
 
 \section{Simplex category an simplicial sets}
In this section we introduce basic definitions and notations
that will be used in the paper. We show that the simplex
category and the interval category are dual.

\begin{definition}
The \textit{simplex category} $\Delta$ has as objects totally
ordered nonempty finite sets and as morphism order preserving
maps.
\end{definition}  

\begin{definition} The \textit{interval category} $\mathcal{I}$
has as objects totally ordered finite sets with at least two
elements and as morphisms order preserving and endpoint
preserving maps.
\end{definition}

Let define $[n]$ and $\{m\}$ to be the following objects:
\[[n]:=\{0,1,...,n\} \in \Delta\] \[ \{m\}:=\{0,1,...,m\}\in
\mathcal{I}.\]

In his paper \cite{jo} Joyal proves that the simplex category
and the interval category are dual. For future use let us give a short proof
here.

\begin{theorem}\label{duality}
The simplex category $\Delta$ and the interval category
$\mathcal{I}$ are dual.
\end{theorem}

\begin{proof}
We will prove the result by constructing explicitly the duality
between the two categories.

Let $G:\mathcal{I}\rightarrow Set$ be the functor
$Hom_{\mathcal{I}}(-,\{1\})$. Note that for each $n$, the set
$Hom_{\mathcal{I}}(\{n\},\{1\})$ is non-empty. We will show that
there is a canonical ordering of this set, such that it can be
regarded as an element of $\Delta$, and we will show that whit
this ordering we can think of $G$ as a functor from
$\mathcal{I}$ to $\Delta$.

We can define an ordering of the set
$Hom_{\mathcal{I}}(\{n\},\{1\})$ as follows:
if $f_1, f_2: \{n\}\rightarrow \{1\}$ are two different maps, we
declare that $f_1<f_2$ if and only if there exists $j$, such
that $f_1(j)<f_2(j)$. Regarded as a functor from $\mathcal{I}$
to $\Delta$, $G$ is defined on objects by $G(\{n\})=[n-1]$.

Let $g:\{n\}\rightarrow\{m\}$ be a morphism in $\mathcal{I}$.
Given the ordering of the hom sets described above,
$G(g):[m-1]\rightarrow[n-1]$ is the following morphism:
\[G(g)(i)=[g^{-1}\{0,1,...,i\}].\]
Note that $i\leq G(g)(j)$ if and only if $g(i)\leq j$.
Therefore, $G(g)$ is well defined order preserving map.

In the other direction, let $F:\Delta\rightarrow Set$ be the
functor $Hom_{\Delta}(-,[1])$. Using the same idea as before, we
see that for each $[n]$, there is a canonical ordering of the
set $Hom_{\Delta}([n] ,[1])$ (which must have two end points -
the constant map at 0 and the constant map at 1), so that we can
regard it as an element of $\mathcal{I}$, and all of the
morphisms induced by $F$ will be morphisms that preserve this
ordering and the end points. This will allow us to think of
$F$ as a functor from $\Delta$ to $\mathcal{I}$.

The functors $F$ and $G$ form a quasi inverse pair.
It is evident from the definitions of $F$ and $G$ that both
compositions $FG$ and $GF$ are identity on objects in
$\mathcal{I}$ and $\Delta$ respectively.
Let $g:\{n\}\rightarrow\{m\}$. Since $g(i)\leq g(i)$, one has
$i\leq G(g)(i)$ and $FG(g)(i)\leq g(i)$. Therefore $FG(g)=g$. A
similar argument shows that $GF(f)=f$.
\end{proof}

\begin{notation}
We write $SSet$ for the category of \textit{simplicial sets},
i.e. the category of functors
\[\Delta^{op}\rightarrow Set.\]
\end{notation}

\section{Endofunctors on the simplex category}
In this section we classify the endofunctors on the simplex
category $\Delta$.
Note that a  functor $T:\Delta\rightarrow\Delta$ induces a
functor $T^*: SSet\rightarrow SSet$ that sends each simplicial
set $X$ to the composition $X\circ T$.

\begin{example}\label{basis functors}
We will refer to the following functors as the \textit{basis}
endofunctors $\Delta\rightarrow\Delta$.
\begin{itemize}
\item The identity functor $Id$, that is identity on both
objects and morphisms;
\item The constant $0$ functor $C_0$ that sends all objects of
$\Delta$ to the object $[0]$, and all morphisms to the identity
$[0]$ morphism.
\item The opposite functor $Op$, that is identity on objects,
and if $f:[n]\rightarrow[m]$ is a morphism, than $Op(f):
[n]\rightarrow [m]$ is the following morphism:
\[Op(f)(k) = m - f(n - k).\] 
\end{itemize}
\end{example}

\begin{definition}
The \textit{concatenation bifunctor} $*:
\Delta\times\Delta\rightarrow \Delta$ is defined as follows:
\begin{itemize}
\item On objects $*([n],[m])=[n]*[m]=[n+m+1]$;
\item On morphisms $(f,g): ([n_1],[m_1])\rightarrow
([n_2],[m_2])$ we have that
\[*(f,g) = f*g:
[n_1+m_1+1]\rightarrow[n_2+m_2+1]\]
is the following map:

\[(f*g)(k)=\begin{cases} f(k) &\mbox{if } 0\leq k \leq n_1 \\
g(k) & \mbox{if } n_1<k\leq n_1+m_1+1. \end{cases} \]
\end{itemize}
\end{definition}

The concatenation bifunctor can be used to define an operation
`$+$' between endofuntors on $\Delta$ in the following way:
\begin{definition} Let $T_1$ and $T_2$ be endofunctors on the
simplex category.
The functor $[T_1+T_2: \Delta\rightarrow \Delta]$ is defined to
be the following composition:
\[\Delta\mathrel{\mathop{\longrightarrow}^{\mathrm{T_1\times
T_2}}}\Delta\times \Delta
\mathrel{\mathop{\rightarrow}^{\mathrm{*}}}\Delta.\]
\end{definition}

Let $\circ$ denote the standard composition of funtors. The
strange notation `$+$' is justified by the following lemma,
whose proof is easy.
 
\begin{lemma}\label{distribution}
Let $T_1$, $T_2$ and $T_3$ be endofunctors on $\Delta$, then
$$(T_1+T_2)\circ T_3=(T_1\circ T_3)+(T_2\circ T_3).$$
\end{lemma}

In the Main Theorem of this section, Theorem \ref{Main Theorem},
we will show that all endofunctors on $\Delta$ can be
represented as a sum under `$+$' of the basis functors defined
in Example \ref{basis functors}. We classify the endofunctors on
$\Delta$ by studying the functors that they induce on $SSet$. Below we will prove that any functor
$T^*: SSet\rightarrow SSet$ that is induced by a functor
$T:\Delta\rightarrow\Delta$, is determined by its value on the
simplicial set $\Delta^1$.
Then, we will examine the relation between $T_1^*(\Delta^1)$,
$T_2^*(\Delta^1)$ and $(T_1+T_2)^*(\Delta^1)$.

 Recall from Section 2 that the categories $\Delta$ and $\mathcal{I}$ are dual. Let \[[\mathcal{L}: \mathcal{I}\rightarrow Set]\] be the forgetful functor. Since $\Delta^1$ is naturally  the functor  $F: \Delta\rightarrow \mathcal{I}$, as constructed in the proof of \ref{duality}. It follows that $T^*(\Delta^1)=\mathcal{L} \circ F \circ T$.   Therefore, we can regard $T^*(\Delta^1)$ as a \emph{simplicial interval} with ordering of the vertices. 

Moreover, using the functors $F$ and $G$, as defined in the proof of Theorem 2.3,  we see that the maps:
\begin{figure}[H]
   \begin{center}
\begin{tikzpicture}\matrix (m) [matrix of math nodes, row sep=0.1em,
column sep=2em, text height=1.5 ex, text depth=0.25ex]
{ Fun(\Delta^{op},\Delta^{op} )& &Fun(\Delta^{op},\mathcal{I})& & Fun(\Delta^{op},\Delta^{op}) \\
T& &F\circ T& & G\circ F\circ T\simeq T \\};
\path[->]
(m-1-1) edge   (m-1-3)
(m-1-3) edge   (m-1-5);
\path[|->]
(m-2-1) edge   (m-2-3)
(m-2-3) edge   (m-2-5);
\end{tikzpicture}
\end{center}
 \end{figure}

 defines an  equivalence of categories. Let 
\[ev: Fun(SSet,SSet)\rightarrow Fun(\Delta^{op}, Set)\]
be the evaluation functor at $\Delta^1$. It factors trough $Fun(\Delta^{op}, \mathcal{I})$, as explained above.  Therefore we have that the following composition is injective

\begin{figure}[H]
   \begin{center}
\begin{tikzpicture}\matrix (m) [matrix of math nodes, row sep=0.1em,
column sep=2em, text height=1.5 ex, text depth=0.25ex]
{ Fun(\Delta^{op},\Delta^{op} )& &Fun(SSet,SSet)& & Fun(\Delta^{op}, \mathcal{I}) \\
T& & T^*& & T^*(\Delta^1)=F\circ T \\};
\path[->]
(m-1-1) edge   (m-1-3)
(m-1-3) edge   (m-1-5);
\path[|->]
(m-2-1) edge   (m-2-3)
(m-2-3) edge   (m-2-5);
\end{tikzpicture}
\end{center}
 \end{figure}

 We see that  an endofunctor $T^*:SSet\rightarrow SSet$ induced by a functor $T:\Delta\rightarrow\Delta$ can be recovered from its value at $\Delta^1$, by first recovering the functor $T=G(T^*(\Delta^1))$. We are thus reduced to classifying the endofunctors on the simplex category by studying the structure of the  simplicial interval $T^*(\Delta^1)$. 

\begin{observation} \label{conc}   Let $T_1$ and $T_2$ be endofunctors on $\Delta$. The simplicial interval induced by  $(T_1+T_2)$ is  the following composition:
 
  \[\Delta\xrightarrow{T_1\times T_2} \Delta\times \Delta \mathrel{\mathop{\rightarrow}^{\mathrm{*}}}\Delta\xrightarrow{Hom_{\Delta}(-,[1])}Set.\]
  
  Let $[n]$ be an object in $\Delta$, such that $T_1([n])=[\alpha]$ and $T_2([n])=[\beta]$. By construction $(T_1+T_2)([n])=[\alpha+\beta+1]$. Hence,
  \[[T_1^*(\Delta^1)]_n=Hom_\Delta([\alpha],[1])\]
  \[[T_2^*(\Delta^1)]_n=Hom_\Delta([\beta],[1])\]
  \[[(T_1+T_2)^*(\Delta^1)]_n=Hom_\Delta([\alpha+\beta+1],[1]).\]

 The maps in $Hom_\Delta([\alpha+\beta+1],[1])$ are formed by concatenating the constant 0 map from $Hom_\Delta([\alpha],[1])$ with any map from $Hom_\Delta([\beta],[1])$, or by concatenating any  map from  $Hom_\Delta([\alpha],[1])$ with the constant 1 map from $Hom_\Delta([\beta],[1])$. Note that in this description the concatenation of the constant 0 map from $Hom_\Delta([\alpha],[1])$ with the constant 1 map from $Hom_\Delta([\beta],[1])$ appears twice. It follows that for any $n$, the $n$-th simplices of $(T_1+T_2)^*(\Delta^1)$ are the union of the $n$-th simplices of $T_1^*(\Delta^1)$ and $T_2^*(\Delta^1)$, where the last simplex of $T_1^*(\Delta^1)$ is identified with the first simplex of $T_2^*(\Delta^1)$. 
 
We conclude the following lemma:
 \end{observation}
\begin{lemma} If $T_1$ and $T_2$ are endofunctors on $\Delta$, then 
 \[(T_1+T_2)^*(\Delta^1)=T_1^*(\Delta^1)\vee T_2^*(\Delta^1).\]
 \end{lemma}
\begin{notation}
Let $T$ be an endofunctor on $\Delta$ and $[n]\in \Delta$, then  $$T(n):=\#T([n])-1.$$
\end{notation}

\begin{lemma}\label{dimsim}
Let $T$ be an endofunctor on $\Delta$ that is not a constant functor. For all objects $[k]\in\Delta$ \[k\leq T(k).\]
\end{lemma}
\begin{proof}
For the sake of contradiction, assume that $T(k)<k$ for some positive integer $k$. Without loss of generality, assume that $k$ is the smallest integer such that $T(k)<k$. Obviously, $k\neq 0$. Note that $T(k)<T(k-1)$, since by assumption
$T(k-1)>k-1$ and $T(k)<k$. Therefore, the maps $T(d_i): T([k-1])\rightarrow T([k])$ is not injective. But this is impossible since $T(d_i)\circ T(s_i)=id$. This contradicts  the assumption that $T(k)<k$.
\end{proof}

\begin{corollary}\label{simplicial interval dimention}
For any $k\geq2$, the simplicial set $T^*(\Delta^1)$ has no non-degenerate $k$-simplices. 
\end{corollary}
\begin{proof}
If $T$ is a constant functor, the statement of the corollary is obvious. Suppose that $T$ is not a constant functor. By \ref{dimsim}, $k\leq T(k)$ for all $[k]\in \Delta$.  Therefore, all maps $T^*(\Delta^1)_k\rightarrow T^*(\Delta^1)_{k-1}$ are surjective. Hence, $T^*{\Delta^1}$ has no non-degenerate $k$-simplices for $k\geq 2$. 
\end{proof}
\begin{lemma}\label{interval}
Let $T:\Delta\rightarrow \Delta$, such that $T([0])=[n]$. The simplicial interval induced by T is  a pointed  union of  $n-1$ elements among \[\Delta^0\sqcup\Delta^0\text{, } \Delta^1,\text{and } [\Delta^1]^{op}.\] 
\end{lemma}
\begin{proof}
By Corollary \ref{simplicial interval dimention}, the simplicial set  $T^*(\Delta^1)$ has no nondegenerate $k$-simplices for $k\geq 2$.  Since $T^*(\Delta^1)$ has an interval structure, then there is at most one directed edge between two consecutive $0$-simplices. 

Suppose that $n=0$. Then $T^*(\Delta^1)$ has two $0$-simplices. If there is no edge between them, $T^*(\Delta^1)=\Delta^0\sqcup\Delta^0$.  If there is an edge from the first $0$-simplex to the second $0$-simplex, then $T^*(\Delta^1)=\Delta^1$.  If there is an edge from the second $0$-simplex to the first $0$-simplex, then $T^*(\Delta^1)=[\Delta^1]^{op}$.  

Similarly, for any $n>0$, $T^*(\Delta^1)$ has $n+2$ ordered $0$-simplices such that there is at most one directed edge between any two consecutive simplices. Therefore, $T^*(\Delta^1)$ is  a pointed  union of  $n-1$ elements among $\Delta^0\sqcup\Delta^0$, $\Delta^1$ and $[\Delta^1]^{op}$. Each of them connects two consecutive simplices. 
\end{proof}

\begin{corollary}\label{the zero functors}
The only endofunctors on $\Delta$  that send the object $[0]$ to itself are the basis functors.
\end{corollary}
\begin{proof}
Let $T:\Delta\rightarrow\Delta$ be a functor, such that $T([0])=[0]$.  By the previous lemma  $T^*(\Delta^1)$ is one of the following: \[\Delta^0\sqcup\Delta^0\text{, } \Delta^1,\text{or } [\Delta^1]^{op}.\]   If $T^*(\Delta^1)=\Delta^0\sqcup\Delta^0$, then $T=C_0$. If $T^*(\Delta^1)=\Delta^1$, then $T=Id$. If $T^*(\Delta^1)=[\Delta^1]^{op}$, then $T=Op$.  
\end{proof}

\begin{theorem}\label{Main Theorem}
 Let $T$ be an endofunctor on the simplex category such that  $T([0])=[n]$. Than $T$ is a sum of $n+1$ of the basis functors.
 \end{theorem}
 \begin{proof}
 Consider the  induced functor $T^*: SSet\rightarrow SSet$. As we have shown, it is completely determined by its value on the simplicial set $\Delta^1$. 
 By assumption $T([0])=[n]$.  Hence,  it is a wedge of  $n+1$ of the following:  $\Delta^0\sqcup \Delta^0$, $\Delta^1$ and $(\Delta^1)^{op}$. \ref{interval}. 
  Recall  that \[(T_i+T_j)^*(\Delta^1)={T_i}^*(\Delta^1)\vee {T_j}^*(\Delta^1).\] By the previous lemma the only functors that sends $[0]$ to $[0]$ are the basis functors. We have that  $Id(\Delta^1)=\Delta^1$, $Op(\Delta^1)=(\Delta^1)^{op}$ and $C_0(\Delta^1)=\Delta^0\sqcup \Delta^0$. We conclude that $T$ is the sum of $n+1$ of the basis functors.
 \end{proof}

\section{Generalized  edgewise subdivisions}
In Section 3 we proved that each endofunctor on the simplex category can be expressed  as a finite sum of the three basis  functors. In this section we study the induced functors on the category of simpicial sets. Now let us classify the endofunctors on $SSet$ that are induced by endofunctors on $\Delta$ and preserve weak equivalences.  
Segal's edgewise subdivision is one example of such endofuntor. 

\begin{definition}
Let $\mathcal{E}$ denotes the  functor $Op+Id$.  \textit{Segal's edgewise subdivision} of a simplicial set $X$ is  the simplicial set  $\mathcal{E}(X)$\cite{segalsub}.   
\end{definition}

 Segal's edgewise subdivision of $\Delta^1$ and $\Delta^2$ are shown in Figure 1 and Figure 2. 

  \begin{figure}[H]
  \begin{center}
\begin{tikzpicture}\matrix (m) [matrix of math nodes, row sep=3em,
column sep=2em, text height=1.5 ex, text depth=0.25ex]
{ 00 & & 01 & & 11 \\ };
\path[->]
(m-1-1) edge   (m-1-3)
(m-1-5) edge  (m-1-3);
\end{tikzpicture}
\end{center}
\caption{Segal's edgewise subdivision of $\Delta^1$}
  \end{figure}

  \begin{figure}[H]\label{ss2}
  \begin{center}
\begin{tikzpicture}\matrix (m) [matrix of math nodes, row sep=2em,
column sep=2em, text height=1.5 ex, text depth=0.25ex]
{ &&11&  &  \\
& 01 & & 12 &\\
00 & & 02 & & 22 \\};
\path[->]
(m-1-3) edge   (m-2-2)
(m-1-3) edge   (m-2-4)
(m-1-3) edge   (m-3-3)
(m-3-1) edge   (m-2-2)
(m-3-1) edge   (m-3-3)
(m-3-5) edge   (m-3-3)
(m-3-5) edge   (m-2-4)
(m-2-2) edge   (m-3-3)
(m-2-4) edge   (m-3-3);
\end{tikzpicture}
\end{center}
 \caption{Segal's edgewise subdivision of $\Delta^2$}
  \end{figure}
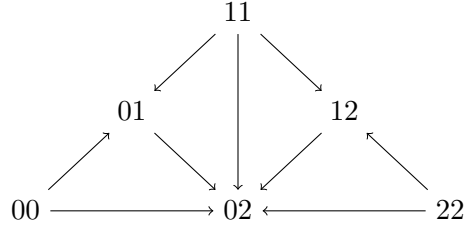

  Obviously, the constant $0$ functor does not induce a functor on $SSet$ that is weak equivalence preserving, while the identity and the opposite functors do. The main result of the section is that a functor on $\Delta$ induces a weak equivalence preserving functor on $SSet$ if and only if it does not contain the constant $0$ functor in its representation as a sum of the basis  functors. 

In his paper \cite{clar} C. Barwick  proves that a functor $T^*$ on the category of simplicial sets preserves weak equivalences (and in fact is a left Quillen functor) if and only if it carries the standard $n$-simplex to a  weakly contractible simplicial set. 

Note that the objects of $\Delta$ can be regarded as categories $[n]$ with objects all integers $i$,  $0\leq i\leq n$ and morphisms $id_i: i\rightarrow i$ and $\beta_i: i\rightarrow i+1$. We show that the category $[n]$ can be embedded in the category $[1]^n$ by a canonical  functor $\eta: [n]\rightarrow [1]^n$, such that there exists a collapsing functor $\mu: [1]^n\rightarrow[n]$ and $\mu\circ\eta = id$.
\begin{construction}
Let $\eta: [n]\rightarrow [1]^n$ be define on objects as follows
\[\eta (m)=\underbrace{0...0}_{n-m+1}\underbrace{1...1}_{m}.\]
If $\beta_i:i\rightarrow i+1$ is the unique morphism from $i$ to $i+1$  in $[n]$, then $\eta(\beta_i)$ is  the unique morphism  from $\underbrace{0...0}_{n-i+1}\underbrace{1...1}_{i}$ to $\underbrace{0...0}_{n-i}\underbrace{1...1}_{i+1}$ in $[1]^n$.

Suppose that $v$ is an object in $[1]^n$. Then $v$ is a sequence of $0$-s  and $1$-s of length $n$. Let $v_i$ denotes the $i$-th term of this sequence. Let $\mu: [1]^n\rightarrow [n]$ be define on objects as follows 
\[\mu (v)=n-\min_ {v_i=1} i+1.\]
Note that in both of the categories $[n]$ and $[1]^n$ if there is a morphism between two objects, it is unique. Therefore, the values of  $\mu$ on morphisms in $[1]^n$ is determined by its values on objects. 

It is evident that $\eta$ is a faithful functor and $\mu$ is a full functor. Moreover, it is easy to chek that $\mu\circ\eta$ is the identity functor on $[n]$. 
\end{construction}
With this, we construct natural transformations $\Phi: \Delta^n\rightarrow [\Delta^1]^n$ and $\Psi: [\Delta^1]^n\rightarrow \Delta^1$ using the functors $\eta$ and $\mu$.

\begin{construction}\label{construction}
First, we construct the map $\Phi:  \Delta^n\rightarrow [\Delta^1]^n$. Recall that $\Delta^n$ is the functor $Hom_\Delta(-,[n])$ and $\Delta^1$ is the functor $Hom_\Delta(-,[1])$. Note that giving $n$ maps $[m]\rightarrow[1]$ is the same as giving one map $[m]\rightarrow[1]^n$.  Therefore $Hom_\Delta([m],[1])^n=Hom_{Cat}([m],[1]^n)$, where $Cat$ is the category of small categories.  The components of the natural transformation $\Phi:  \Delta^n\rightarrow [\Delta^1]^n$ are maps
\[
\begin{array}{@{}r@{\;}r@{\;}c@{\;}l@{}}
    \Phi_m: & Hom_\Delta([m],[n])& \rightarrow & Hom_\Delta([m],[1])^n=Hom_{Cat}([m],[1]^n),   \\
       & f & \mapsto     & \eta\circ f.
  \end{array}
\]

Similarly, the components of the natural transformation $\Psi:  [\Delta^1]^n\rightarrow \Delta^n$ are maps

\[
\begin{array}{@{}r@{\;}r@{\;}c@{\;}l@{}}
    \Psi_m: & Hom_\Delta([m],[1]^n)& \rightarrow & Hom_\Delta([m],[n]),   \\
       & g & \mapsto     & \mu\circ g.
  \end{array}
\]
The naturality of $\Phi$ and $\Psi$ is easy to check.
Note that since $\mu\circ\eta =id$ by construction, then $\Phi\circ\Psi=id$.
\end{construction}

Next, we prove the following stronger result:

\begin{lemma}\label{delta1hom}
The induced endofunctor $T^*: SSet\rightarrow SSet$ on the  category of simplicial sets preserves weak equivalences if and only if it  carries the standard one-simplex $\Delta^1$ to a weakly contractible simplicial set. 
\end{lemma}

\begin{proof}
Given the result in \cite{clar}, it is enough to show that the map $T^*(\Phi):T^*(\Delta^n)\rightarrow [T^*(\Delta^1)]^n$ is a weak equivalence, where $\Phi$ is the map constructed in \ref{construction}. Note that $\Delta^n$ and $[\Delta^1]^n$ are the nerves of the categories $[n]$ and $[1]^n$ respectively. Hence, it is enough to show that the maps $\eta$ and $\mu$ form an inverse pair. By construction $\mu\circ\eta$ is the identity functor on $[n]$.  Note that for all objects $\alpha\in[1]^n$ there is a unique morphism
\[\Gamma_\alpha: \underbrace{...0}_{n-k}\underbrace{1...1}_{k}\rightarrow \underbrace{0...0}_{n-k}\underbrace{1...1}_{k}\]

The collection of the $\Gamma_\alpha$ maps defines a natural transformation $\Gamma$ from the identity functor on $[1]^n$ to the functor $\eta\circ\mu:[1]^n\rightarrow[1]^n$. We conclude that the map $\Psi$ is weak equivalence with homotopy inverce $\Phi$, as constructed in \ref{construction}.

The maps $\Phi$ and $\Psi$ induse the following maps:
\[T^*(\Delta^1)^n\rightarrow T^*(\Delta^n)\rightarrow T^*(\Delta^1)^n\]
We will show that $T^*(\Phi)\circ T^*(\Psi)$ is hopotopic to the identity. We have proved that  there exist homotopy from $\Phi\circ\Psi$ to the $id$ given by:

\begin{figure}[H]\label{hom1} 
  \begin{center}
\begin{tikzpicture}\matrix (m) [matrix of math nodes, row sep=1.5em,
column sep=1em, text height=1.5 ex, text depth=0.25ex]
{ (\Delta^1)^n\times \{0\}& & \\
 &(\Delta^1)^n\times\Delta^1  & (\Delta^1)^n\\
(\Delta^1)^n\times \{1\}& & \\};
\path[right hook->] (m-1-1) edge   (m-2-2);
\path[right hook->] (m-3-1) edge   (m-2-2);
\path[->] (m-2-2) edge   (m-2-3)
(m-1-1) edge [bend left=10] node[above] {$\Phi\circ\Psi$} (m-2-3)
(m-3-1) edge [bend right=10] node[below] {$id$} (m-2-3);
\end{tikzpicture}
\end{center}
  \end{figure}

Note that $T^*$ preserves products, therefore the above diagram becomes:  

\begin{figure}[H]\label{hom2} 
  \begin{center}
\begin{tikzpicture}\matrix (m) [matrix of math nodes, row sep=1.5em,
column sep=1em, text height=1.5 ex, text depth=0.25ex]
{ T^*(\Delta^1)^n & & \\
 &T^*(\Delta^1)^n\times T^*(\Delta^1)  & T^*(\Delta^1)^n\\
T^*(\Delta^1)^n & & \\};
\path[right hook->] (m-1-1) edge   (m-2-2);
\path[right hook->] (m-3-1) edge   (m-2-2);
\path[->] (m-2-2) edge   (m-2-3)
(m-1-1) edge [bend left=10] node[above] {$T^*(\Phi) \circ T^*(\Psi)$} (m-2-3)
(m-3-1) edge [bend right=10] node[below] {$id$} (m-2-3);
\end{tikzpicture}
\end{center}
  \end{figure}

But if $T^*(\Delta^1)$ is weakly contractible the last defines a  homotopy from $T^*(\Phi)\circ T^*(\Phi)$ as desired. 
\end{proof}

\begin{lemma}\label{sumpreserving}
The sum of two weak equivalence preserving endofunctors on $SSet$ is a weak equivalence preserving functor.
\end{lemma}
\begin{proof}
By Lemma \ref{delta1hom}, it is enough to observe that if $T^*_1(\Delta^1)$ and $T^*_2(\Delta^1)$ are weakly contractible, then so is 
\[ (T_1+T_2)^*(\Delta_1)\simeq T_1^*(\Delta^1) \vee T_2^*(\Delta^1). \qedhere \]
\end{proof}
We now conclude:

\begin{theorem}\label{main result}
The functor $T^*: SSet\rightarrow SSet$ preserves weak equivalences if and only if the constant $0$ functor does not appear in the representation of $T$ as a finite sum of the basis functors.    
\end{theorem}

In the rest of section we consider some examples of the generalized edgewise subdivisions of simplicial sets.

\begin{example}
Segal's edgewise subdivision can be applied more than once to a given simplicial sets. Let $\mathcal{E}^n$ denote the endofunctor on $SSet$ that is obtained by applying the Segal's edgewise subdivision $n$ times. This functor is induced by the following endofunctor on $\Delta$   
\[\sum_{i=1}^n Op+ Id\]
 
 The structure of  $\mathcal{E}^2(\Delta^2)$ follows.
  
 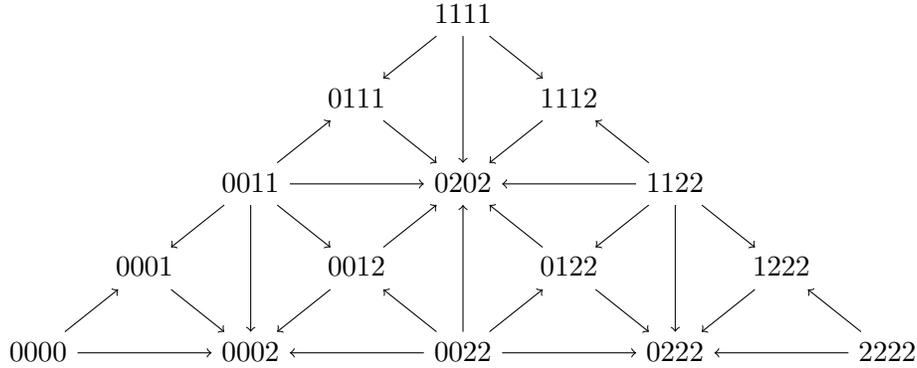
\begin{figure}[H]\label{e2delta2} 
  \begin{center}
\begin{tikzpicture}\matrix (m) [matrix of math nodes, row sep=1.5em,
column sep=1em, text height=1.5 ex, text depth=0.25ex]
{ & && &1111&  & & &  \\
& & & 0111 & & 1112 & & &\\
& & 0011& & 0202 & & 1122 & & \\
& 0001 & & 0012 & &0122 & & 1222 &\\
0000 & & 0002 & &0022 & &0222 & &2222\\};
\path[->]
(m-1-5) edge   (m-2-4)
(m-1-5) edge   (m-2-6)
(m-1-5) edge   (m-3-5)
(m-2-4) edge   (m-3-5)
(m-2-6) edge   (m-3-5)
(m-3-3) edge   (m-2-4)
(m-3-3) edge   (m-3-5)
(m-3-3) edge   (m-4-2)
(m-3-3) edge   (m-4-4)
(m-3-3) edge   (m-5-3)
(m-3-7) edge   (m-2-6)
(m-3-7) edge   (m-3-5)
(m-3-7) edge   (m-4-6)
(m-3-7) edge   (m-4-8)
(m-3-7) edge   (m-5-7)
(m-5-1) edge   (m-4-2)
(m-5-1) edge   (m-5-3)
(m-4-2) edge   (m-5-3)
(m-4-4) edge   (m-5-3)
(m-4-4) edge   (m-3-5)
(m-5-5) edge   (m-5-3)
(m-5-5) edge   (m-5-7)
(m-5-5) edge   (m-3-5)
(m-5-5) edge   (m-4-4)
(m-5-5) edge   (m-4-6)
(m-4-6) edge   (m-3-5)
(m-4-6) edge   (m-5-7)
(m-5-9) edge   (m-5-7)
(m-5-9) edge   (m-4-8)
(m-4-8) edge   (m-5-7);
\end{tikzpicture}
\end{center}
  \caption{Second Segal's edgewise subdivision of $\Delta^2$}
  \end{figure}
\end{example}
\begin{example}
The opposite functor $Op$ is weak equivalence preserving.  Note that $Op$ does not subdivide the simplices, but it reverses the orientation of the arrows.  
\end{example}
\begin{example}
Let denote the endofunctor on $SSet$, that is induced by the endofunctor $Id+Id$ on $\Delta$ by $\mathcal{ID}^2$. By Theorem \ref{sumpreserving} it is also weak equivalence preserving. The structure of  $\mathcal{ID}^2(\Delta^2)$ is given in Figure 4.

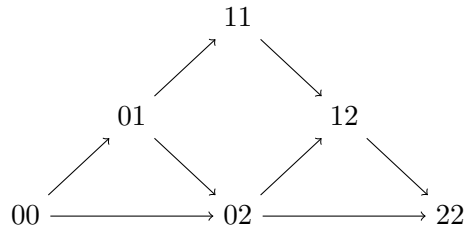
\begin{figure}[H]\label{id2delta2}
   \begin{center}
\begin{tikzpicture}\matrix (m) [matrix of math nodes, row sep=2em,
column sep=2em, text height=1.5 ex, text depth=0.25ex]
{ &&11&  &  \\
& 01& & 12 &\\
00 & & 02 & & 22 \\};
\path[->]
(m-2-2) edge   (m-1-3)
(m-1-3) edge   (m-2-4)

(m-3-1) edge   (m-2-2)
(m-3-1) edge   (m-3-3)
(m-3-3) edge   (m-3-5)
(m-2-4) edge   (m-3-5)
(m-2-2) edge   (m-3-3)
(m-3-3) edge   (m-2-4);
\end{tikzpicture}
\end{center}
  \caption{$\mathcal{ID}^2$ subdivision of $|\Delta^2|$}
  \end{figure}
\end{example}

\nocite{*}
\bibliographystyle{abbrv}
\bibliography{biblio}

\end{document}